\documentclass[11pt]{amsart}
\usepackage[letterpaper, margin=1in]{geometry}

\usepackage{hyperref}
\usepackage{cleveref}
\usepackage[disable,colorinlistoftodos,bordercolor=orange,backgroundcolor=orange!20,linecolor=orange,textsize=scriptsize]{todonotes}

\usepackage{graphicx, url,enumerate}
\usepackage{enumitem}
\usepackage{amssymb,mathtools}

 \renewcommand{\ge}{\geqslant}
 \renewcommand{\le}{\leqslant}

\usepackage{algorithm}
\usepackage{algpseudocode}
\usepackage{color}
\usepackage{tikz}
\usetikzlibrary{3d,calc}
\usepackage{eulervm} 

\theoremstyle{plain}
\newtheorem{theorem}{Theorem}

\newtheorem{lemma}{Lemma}
\newtheorem{corollary}{Corollary}

\theoremstyle{definition}
\newtheorem{definition}{Definition}
\newtheorem{example}{Example}
\newtheorem{problem}{Problem}

\theoremstyle{remark}

\def\<#1>{\langle #1\rangle}

\DeclareMathOperator{\tr}{Tr}
\DeclareMathOperator{\diag}{diag}

\DeclareMathOperator{\V}{V}

\DeclareMathOperator{\supp}{supp}

\DeclareMathOperator{\range}{range}

\newcommand{\N}{\mathbb{N}}
\newcommand{\C}{\mathbb{C}}

\newcommand{\R}{\mathbb{R}}

\newcommand{\x}{\mathbf{x}}

\newcommand{\y}{\mathbf{y}}
\newcommand{\z}{\mathbf{z}}

\renewcommand{\V}{\mathbf{V}}

\newcommand{\cB}{\mathcal{B}}

\newcommand{\cH}{\mathcal{H}}

\newcommand{\cK}{\mathcal{K}}

\newcommand{\cS}{\mathcal{S}}

\newcommand{\bu}{\mathbf{u}}

\newcommand{\bv}{\mathbf{v}}

\newcommand{\rH}{\mathrm{H}}

\newcommand{\rS}{\mathrm{S}}

\newcommand{\0}{\mathbf{0}}
\newcommand{\1}{\mathbf{1}}

\begin{document}
\title{A generalization of Frenkel's formula}
\author{Shmuel  Friedland}
\date{February 15,  2026}
\address{
 Department of Mathematics, Statistics and Computer Science,
 University of Illinois at Chicago, Chicago, IL 60607-7045,
 USA, \texttt{friedlan@uic.edu},
 }
 \subjclass[2010]{
81P47,94A40}

\keywords{
Frenkel's formula,  divergence, operator divergence,  $p$-Schatten operators.
}

\begin{abstract} 
We generalize Frenkel's integral formula for traces of operators to operators.
The resulting formula holds for bounded self-adjoint positive operators and $p$-Schatten  class of compact positive operators.
\end{abstract}
 \maketitle

 \section{Introduction}\label{sec:intro} 
 Denote: $\C^{n\times n}\supset \rH_n\supset \rH_{n,+}\supset \rH_{n,++}$ the complex space of $n\times n$ matrices, the real space of $n\times n$ Hermitian matrices, the cone of positive semidefinite  matrices, and the open set of positive definite matrices.   Let $\rH_{n,+,1}\supset\rH_{n,++,1}$ be the convex set  and the open subset of positive density matrices of order $n$ respetively.   
 For $A,B\in\rH_n$ we denote $\lambda_1(A)\ge \ldots \ge \lambda_n(A)$ the eigenvalues of $A$, and
 $A>B (A\ge B)$ if $A-B\in \rH_{n,++} (A-B\in\rH_{n,+})$. Here, for any $T\in \rH_n$ we denote by $\{T>0\}$ the projection on the positive spectrum of $T$, and 
\begin{equation}\label{defT+}
T_+=T\{T>0\}, \quad T_-=T\{-T>0\}, \quad T=T_+-T_-.
\end{equation}
 
The von Neumann entropy of $\rho\in \rH_{n,+,1}$ is $-\tr \rho\log\rho$, and it is equal to the Shannon entropy of the eigenvalue set of $\rho$.  The Umegaki relative entropy of $\rho$ with respect to $\sigma$, (the divergence of $\rho$ with respect to $\sigma$), is $\tr \rho(\log\rho-\log\sigma)$ \cite{Ume62}.    It is convenient to define the divergence of $A,B\in \rH_{n,+}$ as \cite[eq. (3)]{CGLL25}:
 \begin{equation}\label{DAB}
 D(A\|B):=\tr \big(A(\log A-\log B))-A+B\big).
 \end{equation}
 Furthermore, for $B>0$ let 
 \begin{equation}\label{Dlog}
 D\log[B](A):=\lim_{t\to 0}\frac{\log (B+tA)-\log B}{t}.
 \end{equation}
 
 For numerical computations we have the following well known observations: the von Neumann entropy of $A$ can be computed by finding the eigenvalues of $A$;  $D(A\|B)$ can be found by using the spectral decompositons of $A$ and $B$.    A less known observation is that $D\log[B](A)$ can be found by using the spectral decomposition of $B$ as follows: Find a unitary matrix $U$ such that $U^*B U=\diag(b_1,\ldots,b_n)$ and let $U^*AU=[c_{ij}]$ be Hermitian matrices of order $n$.  Then \cite[Theorem V.3.3]{Bha97} or \cite[eq. (5)]{FG11}
 state:
\begin{equation}\label{FGeq}
D\log[B] (A)=U([c_{ij}\frac{\log b_i-\log b_j}{b_i-b_j}]U^*, \quad \frac{\log b-\log a=b}{b-b}=\frac{1}{b}\textrm{ for } b>0.
\end{equation}
Observe that 
\begin{equation}\label{ADBeq}
\tr B D\log[B] (A)=\tr A, \quad D\log[B] (B)=I_n,
\end{equation}
where $I_n\in\rH_n$ is the identity matrix.
In many theoretical applications physicists prefer integral formulas for the above and similar functions.  A new notable formula that attracted recent activity in quantum information theory (QIT) is Frenkel's integral formula for $D(A\|B)$ \cite{Fre23}:
\begin{equation}\label{Fr}
D(A\|B)=\int_{-\infty}^\infty \frac{dt}{|t|(t-1)^2}\tr \big((1-t)A+tB\big)_-
\end{equation}
See for example \cite{CL25,CGLL25,Jen24,LHC25}.

The aim of this note is  to give the following generalization of \eqref{Fr} to operators:
\begin{equation}\label{Frg}
A(\log A-\log B)-BD\log [B](A)+B=\int_{-\infty}^\infty \frac{dt}{|t|(t-1)^2} \big((1-t)A+tB\big)_-,\quad A\ge 0, B>0.
\end{equation} 
We call the left hand side of this formula the divergence operator:
\begin{equation}\label{divo}
\Delta(A\|B):=A(\log A-\log B)-BD\log [B](A)+B.
\end{equation}
Let
\begin{equation}\label{Og}
O_{\gamma}(A\|B)=(A-\gamma B)_+.
\end{equation}
We show the following equivalent formula analogous to  \cite[eq. (1.6)]{HT24},\cite[Proposition 4.2]{LHC25} and \cite[eq. (10)]{CGLL25}:
\begin{equation}\label{Frg1}
\Delta(A\|B)=\int_1^{\infty} \big(\gamma^{-1} O_{\gamma}(A\|B)+\gamma^{-2}O_{\gamma}(B\|A)\big)d\gamma ,\quad A\ge 0, B>0.
\end{equation} 

It is easy to verify these equalities if $AB=BA$, by assuming that both $A$ and $B$ are diagonal matrices.   The surprising fact that these equalites  hold for any $A,B>0$.
We suspect that there are other integral formulas in QIT where the trace can be replaced by the corresponding operators.
\section{Proof of \eqref{Frg}}\label{sec:prf}
Corollary \ref{corR} yields that for $A,B\in\rH_n$,  the matrix $(A+tB)_+$, hence $(A+tB)_-=(A+tB)_ + -A-tB$, is continuous in $t$ on the real line $\R$.  Furthermore,  there exists as finite number of points $Z\subset\R$, possibly an empty set, such that  $(A+tB)_+$ is analytic on each connected interval of $\R\setminus Z$.  Hence all the integrals in section \ref{sec:intro} make sense, provided that they converge.

As in \cite{Jen24,LHC25,CGLL25} we give an alternative form of the right hand side of \eqref{Frg}.   
\begin{equation}\label{fABeq}
\int_{-\infty}^\infty \frac{dt}{|t|(t-1)^2} \big((1-t)A+tB\big)_-=\int_1^{\infty} \big(\gamma^{-1} O_{\gamma}(A\|B)+\gamma^{-2}O_{\gamma}(B\|A)\big)d\gamma.
\end{equation}
First observe that if $A\ge 0 , B\ge 0$ then $((1-t)A+tB)_-=0$ for $t\in[0,1]$.  Hence,
\begin{equation*}
\int_{0}^\infty \frac{dt}{t(t-1)^2}\big((1-t)A+tB\big)_-=\int_{1}^\infty \frac{dt}{t(t-1)^2}\big((1-t)A+tB\big)_-.
\end{equation*}
Introduce a new variable $\gamma=\frac{t}{t-1}$ on the interval $(1,\infty)$, and recall that $(-C)_-=C_+$ to deduce
\begin{equation*}
\begin{aligned}
&\frac{dt}{t(t-1)^2}\big((1-t)A+tB\big)_-=\frac{dt}{t(t-1)^2}\left((t-1)(-1)\big(A-\frac{t}{t-1)}B\big)\right)_-=\\
&\frac{dt}{t(t-1)}\big(A-\gamma B\big)_+=-\frac{d\gamma}{\gamma}\big(A-\gamma B\big)_+.
\end{aligned}
\end{equation*}
Hence,
\begin{equation*}
\int_{0}^\infty \frac{dt}{t(t-1)^2}\big((1-t)A+tB\big)_-=\int_1^\infty \frac{O_\gamma(A\|B)}{\gamma}d\gamma.
\end{equation*}
On the interval $(-\infty,0)$ we introduce a new variable $\gamma=-\frac{1-t}{t}$.
Observe 
\begin{equation*}
\begin{aligned}
&\big((1-t)A+tB\big)_-=-t\big(B-\gamma A)_+\textrm{ for } t<0\Rightarrow\\ 
&\int_{-\infty}^0 -\frac{dt}{t(1-t)^2}\big((1-t)A+tB\big)_-=\int_1^\infty \frac{d\gamma}{\gamma^2}O_\gamma(B\|A).
\end{aligned}
\end{equation*}
This establishes \eqref{fABeq}.

To prove \eqref{Frg} we follow the proof of Frenkel's integral formula in \cite[Section 4.1]{LHC25}.  Recall \cite[eq. (77)]{LHC25} and \cite[Theorem B.1]{CL25} for $A,B>0$:
\begin{equation}\label{logAB}
\begin{aligned}
&\log A-\log B=\int_1^\infty \big(\{A-\gamma B>0\}-\{B-\gamma A>0\}\big)\frac{1}{\gamma}d\gamma,\\
&D\log [A](B)=\int_0^\infty \{B-\gamma A>0\}d\gamma.
\end{aligned}
\end{equation}
Hence
\begin{equation}\label{eq1}
\begin{aligned}
&A(\log A-\log B)=\int_1^\infty \big(A\{A-\gamma B>0\}-A\{B-\gamma A>0\}\big)\frac{1}{\gamma}d\gamma=u+v-w,\\
&u=\int_1^\infty \big(\gamma^{-1}(A-\gamma B)\{A-\gamma B>0\}+ \gamma^{-2} (B-\gamma A)\{B-\gamma A>0\}\big)d\gamma,\\
&v=\int_1^\infty B\{A-\gamma B>0\}d\gamma,  \quad w=\int_1^\infty \gamma^{-2} B\{B-\gamma A>0\}d\gamma.
\end{aligned}
\end{equation}
Clearly, $u$ is equal to the right hand side of \eqref{fABeq}.  Use the second equality of  \eqref{logAB} to deduce that  
\begin{equation*}
v=BD\log[B](A)-\int_0^1 B\{A-\gamma B>0\}d\gamma.
\end{equation*}
We now evaluate $w$.  Observe that on the interval $(0,1)$ one has the equality 
$\{-A+\alpha B>0\}=I_n -\{A-\alpha B>0\}$ except a finite number of points.  (Here, 
$I_n$ is the identity matrix of order $n$.)
Hence,  by changing to the variable $\alpha=\gamma^{-1}$ for $\gamma\in(1,\infty)$ we obtain
\begin{equation*}
w=\int_0^1 B\{\alpha B -A>0\} d\alpha=\int_0^1  B(I_n-\{A-\alpha B>0\}d\alpha=B-\int_0^1 B\{A-\alpha B>0\}d\alpha.
\end{equation*}
Combine all the above equalities to deduce \eqref{Frg} for $A,B>0$.  Fix $B>0$ and let $A_k>0,k\in N$ and let $\lim_{k\to\infty} A_k=A\ge 0$.  This shows that \eqref{Frg} holds $A\ge 0$ and $B>0$.   One can generalize further as in \cite[Theorem C.1]{CL25}
to matrices to
\begin{equation}\label{gen}
0\le A\le \tau B \textrm{ for some } \tau\ge 1.
\end{equation}
For completeness we bring the following lemma
\begin{lemma}\label{conv} Let $A\in\rH_{n,+},B\in \rH_n$.  Then
\begin{enumerate}
\item The integral $f(t):=\int_1^{t}\gamma^{-2}O_{\gamma}(B\|A) d\gamma$ converges to $F\in\rH_{n,+}$ as $t\to\infty$.
\item Assume that \eqref{gen} holds.  Then 
$$\int_1^\infty \gamma^{-1}O_{\gamma}(A\|B)d\gamma=\int_1^\tau \gamma^{-1}O_{\gamma}(A\|B)d\gamma.$$
\item  Assume that $A,B\ge 0$ and \eqref{gen} does not hold.  Then, $\supp A\not\subseteq \supp B$.  That is, there exists 
$\x\in\C^n, $ such that $ B\x=\0$ and $\x^* A \x>0,  \x^* B\x=0$.  Hence,
\begin{equation}\label{ninft}
\lim_{t\to\infty}\|\int_1^t \big(\gamma^{-1}O_\gamma(A\|B)+\gamma^{-2}O(B\|A)\big)d\gamma\|=\infty.
\end{equation}
\end{enumerate}
\end{lemma}
\begin{proof} (\emph{1}) 
Corollary \ref{corR} yields that $O_\gamma(A\|B), O_\gamma(B\|A)$ are continuous on $\R$.  Hence $f(t)$ is well defined.  Observe next that for $A\ge 0,\gamma\ge 0$:
\begin{equation}\label{BAg}
\begin{aligned}
&0\le O_\gamma(B\|A)=(B-\gamma A)_+=\{B-\gamma A>0\}(B-\gamma A)\{B-\gamma A>0\}=\\
&\{B-\gamma A>0\}B\{B-\gamma A>0\} -\gamma \{B-\gamma A>0\}A\{B-\gamma A>0\}\le\\
&\{B-\gamma A>0\}B\{B-\gamma A>0\}\le \|B\|I_n\Rightarrow \|O_\gamma(B\|A)\|\le \|B\|
\end{aligned}
\end{equation}

Hence,
\begin{equation*}
\ 1\le s\le t \Rightarrow \|f(t)-f(s)\|\le \|B\|\int_s^t\gamma^{-2}d\gamma\le  \|B\|\int_{s}^\infty\gamma^{-2}d\gamma=\frac{\|B\|}{s}.
\end{equation*}
Therefore, $f(m), m\in\N$ is a Cauchy sequence that converges to $F\in\rH_{n,+}$.  Furthermore, the above inequality proves \emph{(1)}.

\noindent\emph{(2)}  Assume \eqref{gen} holds. For $t\ge \tau$ we have
\begin{equation*}
A-\tau B=A-\tau B-(t-\tau)B\le 0\Rightarrow (A-\tau B)_+=0\Rightarrow
\int_1^\infty \gamma^{-1}O_{\gamma}(A\|B)d\gamma=\int_1^\tau \gamma^{-1}O_{\gamma}(A\|B)d\gamma.
\end{equation*}

\noindent\emph{(3)}  Assume that $A,B\ge 0$ and \eqref{gen} does not hold. 
Hence $B$ is singular.  Let $\V\subset \C^n$ be the eigenspace of $B$ corresponding to $0$-eigenvalue.  Suppose that the sesquilinear form $\x^* A\x$ vanishes on $\V$.  As $A\ge 0$ it follows that $A\V=\{\0\}$.  By restricting $A$ and $B$ to the orthogonal complement $\V^\perp$ of $\V$ we deduce that \eqref{gen} holds contrary to ours assumption.  Hence,  there exists $\x\in\V$ such that  $B\x=0$ and $\x^* A\x>0$.  Therefore,
$$0<\x^* A\x=\x^*(A-\gamma B)\x\le  \x^*( A-\gamma B)_+\x\Rightarrow \lim_{t\to\infty}\int_1^{t} \gamma^{-1}\x^* O(A\| B)_\gamma \x d\gamma=\infty.$$
\end{proof}
\begin{theorem}\label{thm1} Assume that $A,B\in \rH_{n,+}$.  Then the following dichotomy holds:
\begin{enumerate}
\item Assume that there exists $\tau\ge 1$ such that $A\le \tau B$.  Then $\Delta(A\|B)\in \rH_{n,+}$, and equality \eqref{Frg1} holds,where 
\begin{equation}\label{zfor}
\Delta(A\|B)=\Delta(A_1\|B_1)\oplus 0,  \quad A_1=A|B\C^n, B_1=B|B\C^n.
\end{equation}
\item Assume that there is no $\tau\ge 1$ such that $A\le \tau B$.   Then the right side 
of \eqref{Frg1} converges to an unbounded nonnegative operator.  More precisely,  there exists $\x\in \C^n$  such that $B\x=0$ and $\x^*A\x>0$.  Hence, \eqref{ninft} holds.
In particular,  in equality \eqref{Fr} both sides are equal to $\infty$.
\end{enumerate}
\end{theorem}
\begin{proof}  Lemma \ref{conv} yields that it is enough to consider the case where $B$ is singular and  $A\le \tau B$ for some $\tau\ge 1$.  Let $\V\subset \C^n$ be the kernel of $B$.  Observe that $A\V=\{\0\}$.  Thus, we can decompose $A=A_1\oplus 0, B=B_1\oplus 0$, where $A_1,B_1$ are the restrictions of $A$ and $B$ to $B\C^n$.   Then $\Delta(A\|B)$ is defined by \eqref{zfor}.
\end{proof}
\section{An infinite dimensional case}\label{sec:op}
Let $\cH$ be an infinite dimensional separable Hilbert space with an inner product $\langle \cdot,\cdot\rangle$,  and the norm $\|\x\|=\sqrt{\langle \x,\x\rangle}$. Denote by $\cB\supset\cS\supset \cS_+\supset \cS_{++}$ the sets of bounded, self-adjoint, nonnegative definite,  and positive definite operators on $\cH$.   For $T\in \cB$ let  $\|T\|=\sup_{\x,\|\x\|=1} \|T\x\|$.   Assume that $T\in \cS$.  As in the finite dimenisional case denote by $\{T>0\}$ the projection on the positive spectrum of $T$.   See \cite{RS80}.
Then the spectral decomposition of $T$ yields the equality \eqref{defT+}.
For $T\in \cB$ let $|T|=\sqrt{T^*T}$.   Then 
\begin{equation}\label{def|T|}
\begin{aligned}
&T=T_+-T_-, \quad |T|=T_+ + T_-,\\
&T_+=\frac{1}{2}(|T|+T), \quad T_-=\frac{1}{2}(|T|-T),\\
&\|T_+\|=\sup_{\|x\|\le 1} \langle \x, T\x\rangle, \quad \|T_-\|=\sup_{\|x\|\le 1} -\langle \x, T\x\rangle,\quad \|T\|=\sup_{\x\in\cH, \|\x\|=1}|\langle \x, T\x\rangle|.
\end{aligned}
\end{equation}
Denote by $\cK_p\supset\cS_p\supset\cS_{p,+}\supset\cS_{p,++}$ the $p$-Schatten space of compact operators, the subspace  of self-adjoint, nonnegative definite,  and positive definite operators with the norm $\| T\|_p=\big(\tr |T|^p\big)^{1/p}, p\in[1,\infty]$.  

It is stated in \cite[Section 2.2]{Jen24} that \eqref{Fr} holds for $A,B\in\cS_{1,+}$ as follows.  Assume first that $\supp A\not\subseteq\supp B$.  Then as in Theorem  \ref{thm1} both sides of \eqref{thm1} are $\infty$.   Assume second that $\supp A\subseteq\supp B$.  As in Theorem \ref{thm1} we can restrict $A$ and $B$ to $\range B$.  Thus, it is enough to discuss the case where $B>0$.  
Thus, we can summarize \cite[Theorem 1]{Jen24} in this form:
\begin{theorem}\label{thm2}  Assume that  $A\in\cS_{1,+}, B\in\cS_{1,++}$.  
Let 
\begin{equation}\label{trbd1}
e_1:=\int_1^{\infty}\gamma^{-1}\|O_{\gamma}(A\|B)\|_1d\gamma +\int_1^{\infty}\gamma^{-2}\|O_{\gamma}(B\|A)\|_1d\gamma.
\end{equation}
Then the following dichotomy holds: Either $e_1<\infty$,
in which case both sides of \eqref{Fr} are finite and equal,  or $e_1=\infty$ and both sides of \eqref{Fr} are $\infty$.
\end{theorem}

We now give the following versions of Theorem \ref{thm1}:
\begin{theorem}\label{thm3}  Assume $\cH$ is infinite dimensional  separable Hilbert space,  and $A\in \cS_+,  B\in \cS_{++}$.
Fix an orthonormal basis $\bv_i, i\in\N$.  
Denote by $P_n$ the projection  on $\V_n=$span$(\bv_1,\ldots,\bv_n)$, and for $n\in\N$ let
\begin{equation}
\begin{aligned}
&A_n=P_nAP_n,\quad B_n=P_nBP_n,\\
&C_n|\V_n:=A_n(\log A_n -\log B_n)|\V_n, D_n=B_nD[B_n](A_n)|\V_n,  \quad C_n|\V_n^\perp= D_n|\V_n^\perp=0.
\end{aligned}
\end{equation}
\begin{enumerate}
\item Let 
\begin{equation}\label{defe}
e:=\int_1^{\infty}\gamma^{-1}\|O_{\gamma}(A\|B)\|d\gamma +\int_1^{\infty}\gamma^{-2}\|O_{\gamma}(B\|A)\| d\gamma, \quad p\in[1,\infty].
\end{equation}
Assume that $e<\infty$.   Then \eqref{Frg1} holds.   The operator divergence $\Delta(A\|B)$ is defined as follows.  
The sequences $\{B_n\}, \{C_n-D_n\}$ of finite dimensional operators  converge in the strong topology to operators $B\in \cS_{++},C-D\in \cS$ respectively, where 
\begin{equation}\label{defAlogAB}
\Delta(A\|B):=C-D+B. 
\end{equation}
If there exists $\tau\ge 1$ such that $A\le \tau B$ then $e<\infty$.  
\item Fix $p\in[1,\infty]$, and assume that $A\in\cS_{p,+}, B\in\cS_{p,++}$ Let
\begin{equation}\label{trbdp}
e_p:=\int_1^{\infty}\gamma^{-1}\|O_{\gamma}(A\|B)\|_pd\gamma +\int_1^{\infty}\gamma^{-2}\|O_{\gamma}(B\|A)\|_p d\gamma, \quad p\in[1,\infty].
\end{equation}
Assume that $e_p<\infty$.  Then \eqref{Frg1} holds.   The left hand side of \eqref{Frg1} is defined as follows.  
The sequences $ \{B_n\},\{C_n-D_n\}$ of finite dimensional operators  converge in the norm $\|\cdot\|_p$ to operators $B,C-D\in\cS_p$, respectively, where \eqref{defAlogAB} holds.
If there exists $\tau\ge 1$ such that $A\le \tau B$ then $e_p<\infty$ in $[1,\infty]$.  
\end{enumerate}
\end{theorem}
\begin{proof}
\emph{(1)}
Kato's inequality \eqref{Kat1} yields that $O_\gamma(A\|B), O_{\gamma}(B\|A)$ are continuous in $\gamma\in\R$.     As $e_{\infty}<\infty$ it follows that $G(t)=\int_1^{t} \big(\gamma^{-1} O_{\gamma}(A\|B)+\gamma^{-2}O_{\gamma}(B\|A)\big)d\gamma$
converges to $G\in\cS_+$ as $t\to\infty$ in the operator norm $\|\cdot\|$.
Let $A_n=P_nAP_n\rH_{n,+}, B_n=P_nBP_n\in\rH_{n,++}$, and
\begin{equation}\label{Gnt}
G_n(t)=\int_1^{t} \big(\gamma^{-1} O_{\gamma}(A_n\|B_n)+\gamma^{-2}O_{\gamma}(B_n\|A_n)\big)d\gamma.   
\end{equation}
The inequality \eqref{p=in} yields
\begin{equation*}
\begin{aligned}
&\|O_\gamma(A_n\|B_n)\|\le \|O_\gamma(A\|B)\|, \quad\|O_\gamma(A_n\|B_n)\|\le \|O_\gamma(A\|B)\|\Rightarrow \\
&\int_1^{\infty} \big(\gamma^{-1} O_{\gamma}(A_n\|B_n)+\gamma^{-2}O_{\gamma}(B_n\|A_n)\big)d\gamma\le e.
\end{aligned}
\end{equation*}
Hence,  $G_n(t)=\int_1^{t} \big(\gamma^{-1} O_{\gamma}(A_n\|B_n)+\gamma^{-2}O_{\gamma}(B_n\|A_n)\big)d\gamma$ converges in the norm  $\|\cdot\|$ to $G_n=\int_1^{\infty} \big(\gamma^{-1} O_{\gamma}(A_n\|B_n)+\gamma^{-2}O_{\gamma}(B_n\|A_n)\big)d\gamma$.   

We claim that $G_n$ converges in the strong topology to $G_n$.  Fix $\y\in\cH, \|\y\|=1$.   Observe that 
\begin{equation*}
\begin{aligned}
&|(G-G_n)\y\|\le \int_1^\infty \big(\gamma^{-1} \|(O_\gamma(A\|B)-O_\gamma(A_n\|B_n))\y\|+\gamma^{-2}\|(O_\gamma(B\|A)-O_\gamma(B_n\|A_n))\y\|\big)d\gamma\le 2e\\
\end{aligned}
\end{equation*}
For $\gamma\in[1,\infty)$ set
\begin{equation*}
\begin{aligned}
&f_n(\gamma)=\gamma^{-1} \|(O_\gamma(A\|B)-O_\gamma(A_n\|B_n))\y\|+\gamma^{-2}\|(O_\gamma(B\|A)-O_\gamma(B_n\|A_n))\y\|, \quad n\in\N,\\
&g(\gamma)=2\big(\gamma^{-1} \|\big(O_{\gamma}(A\|B)\y\|+\gamma^{-2}\|O_{\gamma}(B\|A)\y\|\big).  
\end{aligned}
\end{equation*}
Clearly,  $f_n(\gamma)$, and $g(\gamma)$ are continuous funcrions, and $0\le f_n(\gamma)\le g(\gamma)$ for $\gamma\in[1,\infty]$ and $n\in\N$.
The equality \eqref{p=in} yields that $\lim_{n\to\infty}f_n(\gamma)=0$.   Lebesgue's dominated convergence theorem yields that $G_n,n\to\infty$ converges in the strong topology to $G$.   As $B\in\cS_{++}$ it follows that $B_n$ is positive definite on $\cH_n=P_n\cH$.  Hence,  part \emph{(1)} of Theorem 1 holds for $A_n|\cH_n$ and $B_n|\cH_n$.   As  $G_n,n\to\infty$ converges in the strong topology to $G$, we deduce that $C_n-D_n=B_n$ converges in the strong topology to $G$.
The equality \eqref{stp} yields that $A_n$ and $B_n$ converge in the strong topology to $A$ and $B$ respectively.  Hence,  $C_n-D_n$ converge in the strong topology to $C-D$

Assume now that $A\le \tau B$ for some $\tau\ge 1$.  Then 
$$e':=\int_1^{\infty}\gamma^{-1}\|O_{\gamma}(A\|B)\|d\gamma=\int_1^{\tau}\gamma^{-1}\|O_{\gamma}(A\|B)\|d\gamma<\infty.$$
As in the proof of part \emph{(1)} of Lemma \ref{conv} we deduce that $e'':=\int_1^{\infty}\gamma^{-2}O_{\gamma}(B\|A) d\gamma<\infty$.

\noindent
\emph{(2)}  Fix $p\in[1,\infty]$.   
Let  $L_0,L_1\in\cS_p$, and $L(\gamma):=L_0-\gamma L_1\in \cS_p$ for $\gamma\in(-\infty,\infty)$.  Then
$$\|L(\gamma)\|_p=\big(\|L_+(\gamma)\|_p^p+\|L_-(\gamma)\|_p^p\big)^{1/p}\le \|L_0\|_p+|\gamma|\|L_1\|_p.$$
Recall Rellich's theorem \cite[Theorem 3.9]{Kat80}.  The eigenvalues and the eigenvectors of the pencil $L(\gamma)$ are analytic on $\R$.  Hence each nonzero analytic eigenvalue of $L(\gamma)$ has a countable number of zeros that accumulate to $\pm\infty$.   Therefore $\|L_+(\gamma)\|_p$ is a continuous function on $\R$.  (See Theorem \ref{thmR} for the case where $L_0,L_1$ have finite dimensional range.)
Define
\begin{equation*}
\begin{aligned}
&f_{n,p}(\gamma)=\gamma^{-1} \|(O_\gamma(A\|B)-O_\gamma(A_n\|B_n))\|_p+\gamma^{-2}\|(O_\gamma(B\|A)-O_\gamma(B_n\|A_n))\|_p, \quad n\in\N,\\
&h_{n,p}(\gamma)=\big(\gamma^{-1} \|\big(O_{\gamma}(A_n\|B_n)\|_p+\gamma^{-2}\|O_{\gamma}(B_n\|A_n)\|\big),\\
&g_p(\gamma)=2\big(\gamma^{-1} \|\big(O_{\gamma}(A\|B)\|_p+\gamma^{-2}\|O_{\gamma}(B\|A)\|_p\big).  
\end{aligned}
\end{equation*}
Note that all the above functions are continuous on $\R$.  The equalities \eqref{pl1}
yield that the sequence $h_{n,p}(\gamma),n\in\N$ is an increasing sequence in $n\in\N$ for a fixed $\gamma$, which converges to $g_p(\gamma)/2$.  Furthermore,  $\lim_{n\to\infty}f_{n,p}(\gamma)=0$.  

Assume that $e_p<\infty$. Then $G_p(t)=\int_1^{t} \big(\gamma^{-1} O_{\gamma}(A\|B)+\gamma^{-2}O_{\gamma}(B\|A)\big)d\gamma$
converges to $G_p\in\cS_{p,+}$ as $t\to\infty$ in the  norm $\|\cdot\|_p$.
Let $G_{n,p}=\int_1^{\infty} \big(\gamma^{-1} O_{\gamma}(A\|B)+\gamma^{-2}O_{\gamma}(B\|A)\big)d\gamma\in \cS_{p,+}$.   Lebesgue's dominated convergence theorem yield that $G_{n,p},n\to\infty$ converges in $\|\cdot\|_p$ to $G_p$.    Hence, $C_n-D_n+B_n$ converge to $\Delta(A\|B)$.
Equalities \eqref{pl1} imply that $\lim_{n\to\infty} \|B-B_n\|_p=0$.  Hence,  $C_n-D_n$  
converges to $\Delta(A\|B)-B$.

Assume  that $A\le \tau B$.  Then $e_p<\infty$ as in the proof of part \emph{(1)}.
\end{proof}

\begin{problem}\label{pr1} 
\begin{enumerate}
Assume that $A,B\in\cS_+$ and $A^2\le \alpha^2 B^2$.  Then \eqref{DBi} holds.
Loewner's theorem \cite[Theorem 4.12]{Kat80} yields $A\le \alpha B$.
\item
Is it true that the sequence $D_n,n\in\N$ in Theorem \ref{thm3} converges to $BD[B](A)$ in the strong topology in case \emph{(1)},  and in norm in case  \emph{(2)}?
\item Assume that the assumption of part \emph{(2)} of Lemma \ref{tfo} holds. 
Is it true that the sequence $C_n,n\in\N$ in Theorem \ref{thm3} converges to $A(\log A-\log B)$ in the strong topology in case \emph{(1)},  and in norm in case  \emph{(2)}?

\end{enumerate}
\end{problem}

\section*{Acknowledgment}
The author thanks Gilad Gour for pointing out the references \cite{CGLL25,Fre23},  and Marco Tomamichel  for the reference \cite{HT24}.

\bibliographystyle{plain}

\appendix
\section{Auxiliary results}\label{sec:pac}
\subsection{A finite dimensional case}\label{subsec:fdc}
\begin{definition}\label{dhr}
A Hermitian pencils is defined:
\begin{equation}\label{hpe}
H(z)=A+zB, \quad A,B\in\rH_n.
\end{equation}
A pencil $H(z)$ is called nonsingular if $\det H(z)\not\equiv 0$,  and  $1$-indecomposable if $H(z)$ does not have a constant eigenvector: $H(z)\x=(a+bz)\x, \x\ne 0$.  $H(z)$ is called completely $1$-decomposable, if there exists a unitary $U\in\C^{n\times n}$ such that $H(z)=U\diag(a_1+zb_1),\ldots,a_n+zb_n)U^*$.  
\end{definition}

Recall that $H(z)$ is completely $1$-decomposable if and only if $AB=BA$.
\begin{lemma}\label{dlm}  Assume that $H(z)$ is of the form \eqref{hpe}.
\begin{enumerate}
\item The pencil $H(z)$ is either $1$-indecomposable, completely $1$-decomposable, or the there exists unitary $U\in\C^{n\times n}$ such 
\begin{equation}\label{Hdc1}
\begin{aligned}
&H(z)=U\big(H_1(z)\oplus H_2(z)\big)U^*, \\
&H_2(z)=\diag(a_1+zb_1,\ldots,a_k+zb_k)\big), k\in[n-2], \quad H_1(z))\in \C^{(n-k)\times(n-k)}[z],
\end{aligned}
\end{equation}
where $H_1(z)$ is 1-indecomposable.
\item Assume that $H(z)$ is 1-indecomposable.    Then either $H(z)$ is nonsingular,  or there exists an invertible $P\in\C^{n\times n}$ such that 
\begin{equation}\label{Hdc2}
\begin{aligned}
&H(z)=P\big(\oplus_{j=1}^{k+1} H_j(z)\big)P^*,\\
&H_j(z)=\begin{bmatrix}0_{m_j\times m_j}&L_{m_j}(z)\\L_{m_j}^\top(z)&0_{(m_j+1)\times (m_j+1)}\end{bmatrix}, L_{m_j}=[I_{m_j}\0]+z[\0 \,I_{m_j}]\in\C^{m_j\times(m_j+1)}[z], j\in[k],
\end{aligned}
\end{equation}
where $H_{k+1}=\emptyset$ or $H_{k+1}$ is a nonsingular Hermitian pencil.
In particular, $\det(\lambda I_n-H(z))=\lambda^k p(\lambda,z)$, and $p(0,z)\not \equiv 0$.  Hence,  $H(z)$ has exactly $k$ $0$-eigenvalues, and $n-k$ eigenvalues that vanish
at most $n-k$ points in $\C$.
\end{enumerate}
\end{lemma}
\begin{proof} \emph{(1)} is straightforward.

\noindent
\emph{(2)} Assume that $H(z)$ is $1$-indecomposable singular Hermitian pencil.  Then a variation of the Kronecker canonical form for general pencils gives decomposition 
\eqref{Hdc2} \cite[Theorem 1]{Tho76} or \cite[Theorem 6.1]{LR05}.
Next observe that for $H_j(z)$ has exactly on eigenvector corresponding $0$ eigenvalue.  As $H_j(z)$ is a real symmetric matrix it for a real $z$ it follows that $0$-eigenvalue is of multiplicity one of $H_j(z)$.  Hence, $0$ is an eigenvalue of multiplicity $k$ in $\oplus_{j=1}^{k+1} H_j(z)$.  Therefore,  $\det(\lambda I_n-H(z))=\lambda^k p(\lambda,z)$ where the joint degree of $p(\lambda,z)$ is $n-k$.  Furhtermore $p(0,z)$ is a nonzero polynomial of degree at most $n-k$.  (if $p(0,z)$ vanish identically, we will deduce that $0$ is an eigenvalue of multiplicity at $k+1$.)
Hence the product of the nonvanishing $n-k$ eigenvalues of $H(z)$ vanish at the zeros of $p(0,z)=0$, which is a set of cardinality at most $n-k$.
\end{proof}
\begin{example}\label{exm} Let  $H(z)=\begin{bmatrix}0&1&z\\1&0&0\\z&0&0\end{bmatrix}$.  Then
\begin{equation}\label{hex}
\begin{aligned}
&\det(\lambda I_3 -H(z))=\lambda^3-(1+z^2)\lambda,\\
&\lambda_1(z)=\sqrt{1+z^2}, \bu_1=\frac{1}{\sqrt{2(z^2+1)}}(\sqrt{z^2+1},1,z)^\top,\\
&\lambda_2(z)=0,\bu_2(z)=\frac{1}{\sqrt{z^2+1}}(0,-z,1)^\top, z),\\
&\lambda_3(z)=-\sqrt{1+z^2}, \bu_3=\frac{1}{\sqrt{2(z^2+1)}}(-\sqrt{z^2+1},1,z)^\top.
\end{aligned}
\end{equation}
The points where the two nonzero eigenvalues are zero are $\pm \sqrt{-1}$.
Let $\Omega$ be a simply connected domain in $\C$ containing $\R$ and does not contain $\pm \sqrt{-1}$.  Then the eigenvalues and eignevectors are analytic in $\Omega$.
The eigenvectors are orthonormal on $\R$.  Furthermore, 
\begin{equation*}
\begin{aligned}
&\{H(x)>0\}=\bu_1(x)\bu_1(x)^\top,  H(x)_+=\lambda_1(x)\bu_1(x)\bu_1(x)^\top, \\
&\{-H(x)>0\}=\bu_3(x)\bu_3(x)^\top,  H(x)_1=\lambda_3(x)\bu_3(x)\bu_3(x)^\top.
\end{aligned}
\end{equation*}
The eigenvalues and eigenvectors, except $\lambda_2(z)$,  are two multivalued on $\C\setminus\{\pm \sqrt{-1}\}$.
\end{example}

In what follows we need a following variation of Rellich's theorem \cite[Theorem 4.18.2]{Fribk} :
\begin{theorem}\label{thmR}
Let $H(z)$ be a Hermitian pencil given by \eqref{hpe}.
There exists a simply connected domain $\Omega\subset \C$ such that $\R\subset \Omega, \bar \Omega=\Omega$, and the following conditions hold:
The polynomial $\det (\lambda I_n-H(z))$  splits to $\prod_{j=1}^n (\lambda-\alpha_j(z))$, where each $\alpha_j(z)$ is analytic in $\Omega$ and satisfies $\alpha_j(\bar z)=\overline{\alpha_j(z)}$ for $z\in\Omega$ and $j\in[n]$.  If $\alpha_j(z)$ is not identically zero then each has a finite number of zeros in $\Omega$.
To each eigenvalue $\alpha_j(z)$ corresponds an analytic eigenvector $\bu_j(z)$: $A(z)\bu_j(z)=\alpha_j(z)\bu_j(z)$.    Furthermore $\langle u_i(\bar z), u_j(z)\rangle =\delta_{ij}$ for $i,j\in[n]$ for $z\in\Omega$.  
\item  There exists a finite number of point $Z=\{\zeta_1,\ldots,\zeta_m\}\subset\C\setminus\R$ such that each $\alpha_j(z)$ and $\bu_j(z)$ is a finite multi-valued function on $\C\setminus Z$.    
\end{theorem}  
\begin{proof}
All the claims of the theorem are stated and proved in \cite[Theorem 4.18.12]{Fribk}, except the statement that each $\alpha_j(z)$ that is not identically zero has a finite number of zeros.  This statement follows from Lemma \ref{dlm}.
\end{proof}
\begin{corollary}\label{corR} Let $H(z)$ $n\times n$ nonzero Hermitian pencil.
Assume that $Z=\{\zeta_1,\ldots,\zeta_m\},\subset \R,0\le  m\le n$ be all the real points where the nonzero eigenvalues of $H(x)$ are zero.
\begin{enumerate}
\item The projections $\{H(x)>0\}, \{-H(x)>0\}$ and $H(x)_+, H(x)_-$ are analytic on the intervals $\R\setminus Z$.
\item The matrices  $H_+(x), H_-(x)$ are continuous on $\R$.
\end{enumerate}
\end{corollary}
\subsection{An infinite dimensional case}\label{subsec:fdc}
Assume that $\cH$ is a separable infinite dimensional Hilbert space.   Denote by $\cS'_{++}$ the set positive operaotrs with a bounded inverse.  That is,  $T\in \cS'_{++}\iff \langle T\x,\x\rangle \ge \varepsilon(T)\|\x\|^2$ for some $\varepsilon(T)>0$.

Let $\{T_n,n\in \N\}\subset \cB$.  On $\cB$ we have three topologies:
\begin{enumerate}
\item Norm topology: $\lim T_n\to T\iff \lim_{n\to\infty} \|T_n-T\|=0$.
\item Strong topology: $\lim T_n\stackrel{strong}{\to} T\iff \lim_{n\to\infty} \|T_n\x -T\x\|=0$ for each $\x\in\cH$.
\item Weak topology: $\lim T_n\stackrel{weak}{\to} T\iff \lim_{n\to\infty} \langle \y,T_n\x\rangle=\langle \y,T\x\rangle$ for  each $\x,\y\in\cH$.
\end{enumerate}
Recall that if the sequences $\langle \y,T_n\x\rangle$ ($\{T_n\x\}$) converge for  each $\x,\y\in\cH$ ($\x\in\cH$) then $\sup_{n\in\N}\|T_n\|<\infty$.   Furthermore, there exists 
$T\in \cB$ such that $\lim T_n\stackrel{weak}{\to} T$ ($\lim T_n\stackrel{strong}{\to} T$) \cite[Sec. VI.1]{RS80}.

Let $T\in \cS$.  One associates with $T$ a family of projective-valued measures on the interval $[-\|T\|,\|T\|]$, see \cite[VIII.3,Proposition]{RS80}.  Then, for any bounded Borel function  $f$on $[-\|T\|,\|T\|]$ one can define $f(T)=\int_{-\|T\|}^{\|T\|} f(\lambda)dP_{\lambda}$ such that $\langle \x,f(T)\x\rangle=\int_{-\|T\|}^{\|T\|} f(\lambda)d\langle \x,P_\lambda \x\rangle$.  In particular
\begin{equation}\label{int1}
\begin{aligned}
&T=\int_{-\|T\|}^{\|T\|}\lambda dP_{\lambda},\quad
T_+=\int_0^{\|T\|}\lambda dP_{\lambda}, \quad T_-=-\int^0_{-\|T\|}\lambda dP_{\lambda},\\
&\{T>0\}=\int_{(0,\|T\|]}  dP_{\lambda},  \quad \{-T>0\}\int_{[-\|T\|,0)} dP_{\lambda},\\
&\log T=\int_{\varepsilon(T)}^{\|T\|} \log \lambda dP_{\lambda}, \quad \langle \x,T\x\rangle\ge \varepsilon(T)\|\x\|^2, \x\in\cH,  \varepsilon(T)>0.
\end{aligned}
\end{equation}

However, one prefers integral formulas which do not involve projections.
We start with the following well known integral formula of $\log b, b>0$:
$\log b=\int_0^{\infty}\big((1+x)^{-1}-(b+x)^{-1}\big)dx$ which reduces  to: 
\begin{equation}\label{logB}
\log B=\int_0^\infty \big((Id+x \1)^{-1}-(B+x \1)^{-1})dx, \quad B\in\cS'_{++}.
\end{equation}
(Here $\1$ is the identity operator.)
Observe that the second integral converge as $\|(B+x \1)^{-1}\|\le (\varepsilon(B)+x)^{-1}$ for $B\in \cS'_{++}(\cB)$ and $x\ge -\varepsilon(B)/2$.  Hence,   
\begin{equation*}
\begin{aligned}
&\frac{1}{t}\big(\log (B+tA) -\log B\big)=\frac{1}{t}\int_0^{\infty} \big((B+x \1)^{-1}-(B+tA+x \1)^{-1})dx=\\
&\frac{1}{t}\int_0^{\infty}(B+x \1)^{-1}\big((B+tA+x \1)-(B+x Id)\big)(B+tA+x \1)^{-1}dx=\\
&\int_0^\infty(B+x \1)^{-1}A(B+tA+x \1)^{-1}dx,\quad A\in \cS,  |t|\|A\|\le \varepsilon(B)/2.  
\end{aligned}
\end{equation*}
Thus, we deduce
\begin{equation}\label{int2}
D[B](A)=\lim_{t\to 0} \frac{1}{t}\big(\log (B+tA) -\log B\big)=\int_0^\infty (B+x\1)^{-1}A(B+x\1)^{-1}dx, A\in \cS,  B\in\cS'_{++}.
\end{equation}

Next recall the formula for $\sqrt{a}, a\ge 0$: $\sqrt{a}=\frac{2}{\pi}\int_0^\infty \frac{a}{a+x^2} dx$.   we deduce
\begin{equation}\label{int3}
\begin{aligned}
&|A|=\sqrt{A^2}=\frac{2}{\pi}\int_0^{\infty}{A^2}(A^2+x^2 \1)^{-1}dx=\frac{2}{\pi}\int_0^{\infty}(A^2+x^2 \1)^{-1}A^2dx,\\
&A_+=\frac{1}{2}\big(A+\frac{2}{\pi}\int_0^{\infty}{A^2}(A^2+x^2 \1)^{-1}dx\big)\\
&A_-=\frac{1}{2}\big(-A+\frac{2}{\pi}\int_0^{\infty}{A^2}(A^2+x^2 \1)^{-1}dx\big).
\end{aligned}
\end{equation}
Next observe that since $\|A^2(A^2+x^2 \1)^{-1}\|\le \frac{\|A^2\|}{\|A^2\|+x^2}$ the above integral converges.

It was shown by Kato \cite{Kat73}, basically using the first equality in \eqref{int3}:
\begin{equation}\label{Kat}
\||A|-|B|\|\le \frac{2}{\pi}\|A-B\|\big(2+\log \frac{\|A\|+\|B\|}{\|A-B\|}\big),  \quad A,B\in\cB.
\end{equation}
  For $A,B\in\rS_2$
Araki \cite[Lemma 5.2]{Ara71} showed
\begin{equation}\label{Ara}
\| |A|-|B|\|_2\le \|A-B\|_2.
\end{equation} 
In view of \eqref{def|T|} we deduce:
\begin{equation}\label{Kat1}
\max(\|A_+-B_+\|,\|A_--B_-\|) \le \frac{1}{\pi}\|A-B\|\big(\frac{\pi+4}{2}+\log\frac{\|A\|+\|B\|}{\|A-B\|}\big), \quad A,B\in \cB.
\end{equation}
\subsection{Approximation lemmas}\label{subsec:apr}
\begin{lemma}\label{Pnl} Let $\bv_n,n\in\N$ be an orthonormal basis in $\cH$.  Denote $\V_n=$span$(\bv_1,\ldots,\bv_n)$,  and $P_n:\cH\to \V_n$ the orthogonal projection.
Assume that $T\in\cS$.  Denote 
$$T_n=P_nTP_n, \quad T_{n,+}=(T_n)_+,\quad T_{n,-}=(T_n)_-\in\N.$$
\begin{enumerate}
\item Assume that  $T\in\cS_{\infty}$.
\begin{equation}\label{p=in}
\begin{aligned}
&\|T_{n,+}\|=\max_{\x\in\V_n,\|\x\|\le 1} \langle \x,T\x\rangle\le \|T_{n+1,+}\|\le \|T_+\|, \quad n\in\N,\\
&\lim_{n\to\infty} \|T-T_n\|=0,  \lim _{n\to\infty} \|T_+-T_{n,+}\| =0,  \quad \lim _{n\to\infty} \|T_--T_{n,-}\| =0.
\end{aligned}
\end{equation}
\item Assume that $p\in[1,\infty)$ and $T\in \cS_p$.  Then
\begin{equation}\label{pl1}
\begin{aligned}
&\|T_{n,+}\|_p\le \|T_{n+1,+}\|_p\le \|T_+\|_p,   \quad \|T_{n,-}\|_p\le \|T_{n+1,-}\|_p\le \|T_-\|_p,\quad n\in\N,\\
&\lim_{n\to\infty} \|T-T_n\|_p=0,  \lim _{n\to\infty} \|T_+-T_{n,+}\|_p =0,  \quad \lim _{n\to\infty} \|T_--T_{n,-}\|_p =0.
\end{aligned}
\end{equation}
\item Assume that $T\in \cS$.  Then the following sequences converge in the strong topology
\begin{equation}\label{stp}
\lim_{n\to\infty} T_n= T, \quad \lim_{n\to \infty} T_{n,+}=T_+, \quad \lim_{n\to\infty} T_{n,-}=T_-.
\end{equation}
\end{enumerate}
\end{lemma}
\begin{proof}\emph{(1)} Let $T\in\cS_{\infty}$.   Then $T=T_+-T_-$, where $T_+\cH=\cH_+, T_-\cH=\cH_-$, $\cH_+$ and $\cH_-$ are two orthogonal subspaces in $\cH$.   We first consider the generic case where $\cH=\cH_+\oplus \cH_-$, and $\cH_+,\cH_-$ are infinite dimensional.  Assume that $\cH_+$ and $\cH_-$ have orthonormal bases $\x_i,i\in\N$ and $\y_i,i\in\N$ respectively.  Hence,
\begin{equation}\label{T-+ex}
\begin{aligned}
&T_+=\sum_{i=1}^\infty \mu_i \x_i\x_i^*,  \mu_1\ge\cdots\ge\mu_n >0, \langle \x_i,\x_j\rangle=\delta_{ij},i,j,n\in\N, \lim_{n\to\infty}\mu_n=0,\\
&T_-=\sum_{i=1}^\infty \nu_i \y_i\y_i^*,  \nu_1\ge\cdots\ge\nu_n >0, \langle \y_i,\y_j\rangle=\delta_{ij},i,j,n\in\N, \lim_{n\to\infty}\nu_n=0
\end{aligned}
\end{equation}
Observe that the above expansions converge in norm.  This follows from the observation:
\begin{equation}
\|\sum_{i=j}^{k} \mu_i\x_i\x_i^*\|=\mu_j \textrm{ for } k\ge j.
\end{equation}

Next observe that the first equality in \eqref{p=in} is obvious.  (Note that $T_{n,+}$ can be $0$.) Hence, $\|T_{n,+}\|\le \|T_{n+1,+}\|$.  Next observe that $T_{n,+}=Q_nT_nQ_n$, where $Q_n:\V_n\to\V_n$ is a projection.  
Thus 
\begin{equation}\label{T+nin}
T_{n,+}=Q_n(P_nT_+P_n)Q_n-Q_n(P_nT_-P_n)Q_n\le Q_n(P_nT_+P_n)Q_n\Rightarrow \|T_{n,+}\|\le \|T_+\|.
\end{equation}
We now prove the equality $\lim_{n\to\infty}\|T_+-P_nT_+P_n\|=0$.  
Define
\begin{equation}\label{defRN}
R_N=\sum_{i=N+1}^\infty \mu_i\x_i\x_i^*\in\cS_{\infty,+}, \quad N\in\N.
\end{equation}
Fix $\delta>0$.  Choose $N(\delta)$ such that $\mu_N(\delta)<\delta/4$.  Then
\begin{equation}\label{tin}
\begin{aligned}
&\|R_{N(\delta}\|<\delta/4\Rightarrow \|P_nR_{N(\delta)}P_n\|< \delta/4 \Rightarrow\\
&\|T_+-P_nT_+P_n\|< \|\sum_{i=1}^{N(\delta)}\mu_i\x_i\x_i^*-P_n\big(\sum_{i=1}^{N(\delta)}\mu_i\x_i\x_i^*\big)P_n\|+\delta/2.
\end{aligned}
\end{equation}
As $P_n(\x_i\x_i^*)P_n=(P_n\x_i)(P_n\x_i)^*\to \x_i\x_i^*$ in norm we deduce that there exists $M(\delta)$ such that for $n>M(\delta)$ we have $\|T_+-P_nT_+P_n\|<\delta$.  Similarly, $\lim_{n\to\infty}\|T_--P_nT_-P_n\|=0$.  This show that $\lim_{n\to\infty}\|T-T_n\|=0$.   Kato's inequality \eqref{Kat1} yields the last two inequalities in \eqref{p=in}.  The other cases are proved similarly

\noindent \emph{(2)} Assume that $p\in[1,\infty)$.   Assume first the generic case as in part \emph{(1)}.  Then $\|T\|_p=(\|T_+\|^p_p+\|T_-\|^p_p)^{1/p}$.
The inequality $\|T_{n,+}\|_p\le \|T_{n+1,+}\|_p$ follows from the Cauchy interlacing theorem.  Indeed,  we can assume that $T_n\in\rH_n, T_{n+1}\in \rH_{n+1}$, and $T_n$ is obtained from $T_{n+1}$ by deleting the last row and column.
Assume that $\lambda_1(T_{n+1})\ge \ldots\ge \lambda_m(T_{n+1})\ge 0>\lambda_{m+1}(T_{n+1})$, where $m\le n+1$.  Cauchy interlacing theorem yields:
\begin{equation}\label{cit}
\lambda_i(T_{n+1})\ge \lambda_i(T_n)\ge \lambda_{i-1}(T_{n+1}), \quad i\in[n].
\end{equation}
Hence, $\lambda_m(T_n)\in [\lambda_{m}(T_{n+1}),\lambda_{m+1}(T)]$.   Suppose first that $\lambda_{m}(T_n)\in [\lambda_{m}(T_{n+1},0]$.   Then
\begin{equation*}
\begin{aligned}
&\lambda_i(T_{n+1,+})=\lambda_i(T_{n+1}),  i\in[m], \quad \lambda_i(T_{n+1,+})=0 \textrm{ for } i>m,\\
&\lambda_i(T_{n,+})=\lambda_i(T_{n}),  i\in[m], \quad \lambda_i(T_{n,+})=0 \textrm{ for } i\ge m+1\Rightarrow\\
&\|T_{n,+}\|_p=\big(\sum_{i=1}^{m-1}\lambda_i^p(T_n)\big)^{1/p}\le \|T_{n+1,+}\|_p=\big(\sum_{i=1}^{m}\lambda_i^p(T_{n+1})\big)^{1/p}
\end{aligned}
\end{equation*}
Similar arguments yield that $\|T_{n,-}\|_p\le \|T_{n+1,-}\|_p$.  The above inequalities hold when $\lambda_{m}(T)\in (0, \lambda_{m+1}(T_{n+1})]$.   

Our assumption that $T_+$ has infinite number of eigenvalues and the fact that $\lim_{n\to \infty}\|T_+-T_{n,+}\|=0$ yields the following fact.  Fix $k\ge 1$ and consider the sequence $\lambda_i(T_n),n\in\N$ for $i\in[k]$.   Then each $\lambda_{i}(T_n),n\ge k$  is a nondecreasing sequence that converges to $\mu_i$ for $i\in[k]$. 
In particular, we deduce that 
\begin{equation}\label{bin1}
\begin{aligned}
&\sum_{i=j}^n\lambda_i^p(T_{n,+})\le \sum_{i=j}^{n}\mu_i^p(T_+), \quad j\in[n], n\in\N \Rightarrow\\ 
&\|T_{n,+}\|_p\le\big(\sum_{i=j}^{n}\mu_i^p(T_+)\big)^{1/p}\le \|T_+\|_p.
\end{aligned}
\end{equation}
Similar results hold for $\|T_{n,-}\|_p$.

We now show that $\lim_{n\to\infty}\|T_+-T_{n,+}\|_p=0$.   Fix $\delta>0$.
Let $R_N$ be defined by  \eqref{defRN}.  Clearly, $R_N\in\cS_{p,+}$ as
as $\|R_{N}\|_p=\big(\sum_{i=N+1}^\infty \mu_i^p\big)^{1/p}$.
Choose $N(\delta)$ such that  $\|R_N\|_p<\delta/4$.  Clearly $P_nR_NP_n\ge 0$.
Hence,  the first set of inequalities of \eqref{pl1} yield $\|P_nR_NP_n\|_p\le \|R_N\|_p<\delta/4$.   As in $\eqref{tin}$ we obtain
\begin{equation*}
\|T_+-P_nT_+P_n\|_p< \|\sum_{i=1}^{N(\delta)}\mu_i\x_i\x_i^*-P_n\big(\sum_{i=1}^{N(\delta)}\mu_i\x_i\x_i^*\big)P_n\|_p+\delta/2.
\end{equation*}

It is left to show rank-two operator $\x_i^*\x_i-(P_n\x_i)(P_n\x_i)^*$ converges in the operator norm $\|\cdot \|_p$ to zero, which is trivial.  Hence, $\|T_+-T_{n,+}\|_p<\delta$
for $n>M(\delta)$.   This proves the last two equalities in \eqref{p=in}, which imply that $\lim_{n\to\infty}\|T-T_n\|_p=0$.  Other cases are proved similarly.

\noindent 
\emph{(3)}  Clearly, $P_n\stackrel{strong}{\to} \1$.  Recall \cite[Sec. VI, Problem (6d)]{RS80} that if $A_n,B_n\in\cS$ converge in the strong topology to $A,B\in\rS(\cB)$ respectively, then $A_nB_n \stackrel{strong}{\to} AB$.  Hence
\begin{equation*}
T_{n,+}-T_{n,-}=P_nTP_n\stackrel{strong}{\to} T,=T_+-T_- \quad, (P_nTP_n)^2 \stackrel{strong}{\to} T^2.  
\end{equation*}
Problem 14(b) in  \cite[Sec. VI]{RS80} states that if $A_n\in\rS_+(\cB)$ converges strongly to $A\in\rS_+(\cB)$ then $\sqrt{A_n}\stackrel{strong}{\to} \sqrt{A}$.
Hence
\begin{equation*}
T_{n,+}+T_{n,-}=\sqrt{T_n^2}\stackrel{strong}{\to}\sqrt{T^2}=T_++T_-.
\end{equation*}
These two equalities prove the last two equalities in \eqref{stp}.
\end{proof}
\subsection{Integral representation of two operators}\label{subsec:iro}
In this subsection we give two convergent formulas for the operators appearing in 
$\Delta(A\|B)$:
\begin{lemma}\label{tfo} Assume that $A\in \cS,B\in \cS_{+}$, and $A^2\le \alpha^2 B^2$ for some $\alpha>0$.
\begin{enumerate}
\item The following integral formula for $BD[B](A)$ converges, and its norm is bounded from above:
\begin{equation}\label{DBi}
\begin{aligned}
&BD[B](A)=\int_0^\infty B(B+x\1)^{-1}A(B+x\1)^{-1}dx,\\
&\|BD[B](A)\|\le \alpha \|B\|.
\end{aligned}
\end{equation}
\item Suppose that $A\in \cS_+$. Then the following  integral formula converges 
\begin{equation}\label{tfo1}
A(\log A-\log B)=\int_0^\infty \big( A(B+x \1)^{-1}-A(A+x \1)^{-1}\big)dx
\end{equation}
if one of the following conditions hold:
\begin{enumerate}
\item The inequality $B^2\le \beta^2A^2$ holds for some $\beta\ge \frac{1}{\alpha}$. Then
\begin{equation}\label{tfo2}
\|A(\log A-\log B)\|\le\alpha(1+\beta)\|A\|\|B\|\frac{\log\|B\|-\log\|A\|}{\|B\|-\|A\|}.
\end{equation}
\item The inequality $(A-B)^2\le \beta^2 A^2$ holds.  Then
\begin{equation}\label{tfo3}
\|A(\log A-\log B)\|\le \alpha\beta\|A\|\|B\|\frac{\log\|B\|-\log\|A\|}{\|B\|-\|A\|}.
\end{equation}
\end{enumerate}
\end{enumerate}
\end{lemma}
\begin{proof}Since $A^2\le \alpha^2 B^2$, we deduce that $\ker B\subseteq\ker A$.
Hence,  without loss of generality we can assume that $B\in\cS_{++}$.   
For $\varepsilon\in[0,1]$ let $A_{\varepsilon}=A+\varepsilon \1,  B_{\varepsilon}=A+\varepsilon \1$.  

\noindent
\emph{(1)} 
Fix $x>0$, and $\varepsilon\in(0,1]$. 
Clearly,  $A^2\le \alpha^2 B^2_{\varepsilon}$.
The equality \eqref{int2} yields:
\begin{equation*}
BD[B_{\varepsilon}](A)=\int_0^\infty B(B+(x+\varepsilon)\1)^{-1}A(B+(x+\varepsilon)\1)^{-1} dx.
\end{equation*}
Observe that for $x>0$
\begin{equation*}
\begin{aligned}
&\|B(B+x\1)^{-1}A(B+x\1)^{-1}\|\le \|B(B+x\1)^{-1}\| \|A(B+x\1)^{-1}\|\le\\ 
&\|B(B+x\1)^{-1}\|\|(B+x\1)^{-1}A^2(B+x\1)^{-1}\|^{1/2},\\
&(B+x\1)^{-1}A^2(B+x\1)^{-1}\le (B+x\1)^{-1}\alpha^2B^2(B+x\1)^{-1}\Rightarrow\\
&\|(B+x\1)^{-1}A^2(B+x\1)^{-1}\|\le \alpha^2\|(B+x\1)^{-1}B^2(B+x\1)^{-1}\|=\alpha^2\|B(B+x\1)^{-1}\|^2,\\
&\|B(B+x\1)^{-1}\|\le\max_{t\in[0,\|B\|]}\frac{t}{t+x}=\frac{\|B\|}{\|B\|+x}.
\end{aligned}
\end{equation*}
(Observe that this inequality is equality.)
Hence,
\begin{equation*}
\begin{aligned}
&\|B(B+x\1)^{-1}A(B+x\1)^{-1}\|\le \frac{\alpha\| B\|^2}{(\|B\|+x)^2}\Rightarrow\\
&\|BD[B_{\varepsilon}](A)\|\le\int_0^\infty  \|B(B+(x+\varepsilon)\1)^{-1}A(B+(x+\varepsilon)\1)^{-1}\|dx\le\\
&\int_0^\infty \frac{\alpha\| B\|^2}{(\|B\|+x+\varepsilon))^2}dx=\alpha \frac{\|B\|^2}{\|B\|+\varepsilon}<\alpha\|B\|.
\end{aligned}
\end{equation*}

These inequalities shows that the integral in \eqref{DBi} converges.   The Lebesgue's dominated convergence theorem yields that for $\varepsilon_n=1/n,n\in\N$ the integrals for $BD[B_{1/n}](A)$ converge in the norm $\|\cdot\|$ to the integral in \eqref{DBi}.
(See the proof of Theorem \ref{thm3}.)  This proves \eqref{DBi}.

\noindent\emph{(2)}  
Use \eqref{logB} to deduce
\begin{equation*}
\begin{aligned}
&A(\log A_{\varepsilon}-\log B_{\varepsilon})=\int_0^\infty A\big((B_{\varepsilon}+x\1)^{-1}-(A_{\varepsilon}+x\1)^{-1}\big)dx=\\
&\int_0^\infty A\left(B_{\varepsilon}+x\1)^{-1}\big((A_{\varepsilon}+x\1)-(B_{\varepsilon}+x\1)(A_{\varepsilon}+x\1)^{-1}\right)dx=\\
&\int_0^\infty A(B_{\varepsilon}+x\1)^{-1}(A-B)(A_{\varepsilon}+x\1)^{-1} dx=F(\varepsilon)-G(\varepsilon),\\
&F(\varepsilon)=\int_0^\infty A(B_{\varepsilon}+x\1)^{-1}A(A_{\varepsilon}+x\1)^{-1} dx , \quad G(\varepsilon)=\int_0^\infty A(B_{\varepsilon}+x\1)^{-1}B(A_{\varepsilon}+x\1)^{-1} dx.
\end{aligned}
\end{equation*}

We claim that 
\begin{equation}\label{tfo4}
\lim_{n\to\infty}\|F(1/n)-F(0)\|=0, \quad lim_{n\to\infty}\|G(1/n)-G(0)\|,
\end{equation}
if the condition (a) holds, and $lim_{n\to\infty}\|\big(F(1/n)+G(1/n)\big)-\big(F(0)+G(0)\big)\|=0$ if the condition (b) holds.
Assume first the condition (a).
Use the arguments of part \emph{(1)} to deduce
\begin{equation*}
\begin{aligned}
&\|A(B_{\varepsilon}+x\1)^{-1}A(A_{\varepsilon}+x\1)^{-1}\|\le\|A(B+(x+\varepsilon)\1)^{-1}\|\| A(A+x+\varepsilon\1)^{-1}\|\le\\
&\left(\frac{\alpha\|B\|}{\|B\|+x+\varepsilon}\right)\left(\frac{\|A\|}{\|A\|+x+\varepsilon}\right)
\end{aligned}
\end{equation*}
Hence,  
\begin{equation*}
\begin{aligned}
&\|F(\varepsilon)\|\le \int_0^\infty \|A(B_{\varepsilon}+x\1)^{-1}A(A_{\varepsilon}+x\1)^{-1} \|dx\le\\ 
&\alpha\|A\|\|B\| \int_0^{\infty} \frac{dx}{(\|B\|+x+\varepsilon)(\|A\|+x+\varepsilon)}dx=\\
&\alpha\|A\|\|B\|\frac{\log(\|B\|+\varepsilon)-\log(\|A\|+\varepsilon)}{\|B\|-\|A\|}\le 
\alpha\|A\|\|B\|\frac{\log(\|B\|+1)-\log(\|A\|+1)}{\|B\|-\|A\|}
\end{aligned}
\end{equation*}
The arguments in part \emph{(1)} yield that $F_{\varepsilon}$ is well defined for $\varepsilon\in(0,1]$.  To estimate $G(\varepsilon)$ we use the inequality $B^2\le \beta^2 A^2$  to show
\begin{equation*}
\begin{aligned}
&\|A(B_{\varepsilon}+x\1)^{-1}B(A_{\varepsilon}+x\1)^{-1}\|\le\|A(B+(x+\varepsilon)\1)^{-1}\|\| B(A+x+\varepsilon\1)^{-1}\|\le\\
&\left(\frac{\alpha\|B\|}{\|B\|+x+\varepsilon}\right)\left(\frac{\beta\|A\|}{\|A\|+x+\varepsilon}\right)\Rightarrow\\
&\|G(\varepsilon)\|\le \alpha\beta\|A\|\|B\|\frac{\log(\|B\|+\varepsilon)-\log(\|A\|+\varepsilon)}{\|B\|-\|A\|}\le 
\alpha\beta\|A\|\|B\|\frac{\log(\|B\|+1)-\log(\|A\|+1)}{\|B\|-\|A\|}
\end{aligned}
\end{equation*}
The equalities \eqref{tfo4} follow from Lebesgue's dominated convergence theorem.
This proves (2a).

Assume now the condition (b) $(A-B)^2\le \beta^2 A^2$.  Then
\begin{equation*}
\begin{aligned}
&\|A(B_{\varepsilon}+x\1)^{-1}(A-B)(A_{\varepsilon}+x\1)^{-1}\|\le\|A(B+(x+\varepsilon)\1)^{-1}\|\| (A-B)(A+(x+\varepsilon)\1)^{-1}\|\le\\
&\left(\frac{\alpha\|B\|}{\|B\|+x+\varepsilon}\right)\|(A+(x+\varepsilon)\1)^{-1}(A-B)^2(A+(+\varepsilon)x\1)^{-1}\|^{1/2}\le\\
&\left(\frac{\alpha\|B\|}{\|B\|+x+\varepsilon}\right)\left(\frac{\beta\|A\|}{\|A\|+x+\varepsilon}\right)
\end{aligned}
\end{equation*}
Hence, 
\begin{equation*}
\begin{aligned}
&\|F(\varepsilon)-G(\varepsilon)\|\le \int_0^\infty\|A(B_{\varepsilon}+x\1)^{-1}(A-B)(A_{\varepsilon}+x\1)^{-1} \|dx\le\\
&\alpha\beta\|A\|\|B\|\frac{\log(\|B\|+\varepsilon)-\log(\|A\|+\varepsilon)}{\|B\|-\|A\|}\le 
\alpha\beta\|A\|\|B\|\frac{\log(\|B\|+1)-\log(\|A\|+1)}{\|B\|-\|A\|}.
\end{aligned}
\end{equation*}
The arguments of the proof of the case (2a) yield (2b).
\end{proof}

\end{document}